\documentclass[preprint,12pt,authoryear]{elsarticle}

\usepackage{amsmath,amsthm}     
\usepackage{graphicx}     
\usepackage{url}
\usepackage{amsfonts} 

\usepackage{multicol}

\usepackage{amssymb}

\usepackage{amssymb}
\usepackage{amsmath}
\newtheorem{theorem}{Theorem}
\newtheorem{lema}{Lemma}
\newtheorem{proposition}{Proposition}
\theoremstyle{definition}

\newtheorem*{remark}{Remark}

\newcommand{\C}{\mathbb{C}}
\newcommand{\D}{\mathbb{D}}
\newcommand{\T}{\mathbb{T}}

\begin{document}

\begin{frontmatter}



\title{Embedding $H^\infty(\D)$ into $L^\infty(\T)$: a proof without non-tangential limits} %% Article title


\author{Mario P. Maletzki\footnote{The author  has received funding from the following source: CIACIF/2021/479. Programa Fons Social Europeu Plus (FSE+) Comunitat Valenciana
2021-2027.}} %% Author name

\affiliation{organization={IMAC},%Department and Organization
            addressline={Universitat Jaume I }, 
            city={Castellón de la Plana},
            postcode={12071}, 
            %state={},
            country={Spain}}

%% Abstract
\begin{abstract}
The purpose of this note is to show in an accessible and self-contained way  the existence of an isometric algebra embedding from  $H^\infty(\D)$ into $L^\infty(\T)$, without appealing to Fatou's classical theorem on non-tangential limits of analytic functions, and relying only on results from complex and functional analysis that are typically covered in a standard undergraduate course. 
\end{abstract}



%% Keywords
\begin{keyword}
Isometric embedding \sep Poisson kernel \sep Hardy space


\end{keyword}

\end{frontmatter}




%In this appendix we introduce the Poisson integral and prove the existence of an isometric algebra homomorphism from $H^\infty(\D)$ into $L^\infty(\T)$ without appealing to Fatou's famous theorem on non-tangential limits of analytic funcions.  

\section{Introduction}

Beyond its intrinsic interest, the existence of an isometric embedding $H^\infty(\D) \hookrightarrow L^\infty(\T)$ has proven to be useful in the study of the classical Hardy space $H^\infty(\D)$ of bounded analytic functions on $\D$, and for its applications (such as the identification of the \emph{Shilov boundary} of $H^\infty(\D)$ or the study interpolating sequences for this space in terms of Poisson  integrals) the reader is referred to \cite{Garnett1981}. For background and notation on functional and complex analysis see \cite{ConwayCFA} and \cite{RudinRCA} respectively, and for basic facts  on the theory of Hardy spaces see for instance \cite{DurenHP}. %, where in particular it is used to identify its Shilov boundary, and allows to study interpolating sequences for this space in terms of Poisson  integrals (see ).

We first recall that the  Poisson kernel in the unit disk is given by $$P_r(\theta):=\frac{1-r^2}{1-2r\cos(\theta)+r^2}=\frac{1-|re^{i\theta}|^2}{|1-re^{i\theta}|^2}, \quad \mbox{ for } 0\leq r<1,$$ and that it is \emph{an approximate identity for $L^1(\T)$}, which means that viewing the Poisson kernel as a net $(P_r)_r\subset L^1(\T)$, we have that $||P_r\ast f-f||_1\rightarrow0$ as $r\rightarrow1$.  The last fact follows from the following well-known properties: %by saying that it is \emph{an approximate identity for $L^1(\T)$}, that is, since 

%For every $0\leq r<1$,
\begin{enumerate}
    \item[$i)$] $P_r(\sigma)$ is decreasing on $0<\sigma<\pi$,
    \item[$ii)$] $P_r(\sigma)\geq 0$ for every $\sigma$,
    \item[$iii)$] $P_r(\sigma)=P_r(-\sigma)$ for every $0<\sigma<\pi$,
    \item[$iv)$] It holds $$\frac{1}{2\pi}\int_{-\pi}^\pi P_{r}(\sigma)\ d\sigma=1.$$
    \item[$v)$] For every $\delta>0$, $$\lim_{r\rightarrow 1} \sup_{|t|\geq \delta}\{P_r(t)\}=0.$$
\end{enumerate}

%And also, \begin{enumerate}

%\end{enumerate}

Given a   harmonic function $u$ in a disk $D(a,R)$ of the complex plane  that  is also continuous in $\overline{D(a,R)}$, we have that  for each $z=a+\rho e^{i\sigma}\in D(a,R)$ the following equality holds  \begin{equation}\label{Poisson-har}
    u(z)=\frac{1}{2\pi}\int_{-\pi}^\pi u(e^{it})P_{\frac{\rho}{R}}(\sigma-t)\ dt,\end{equation} which is called the \emph{Poisson integral formula for harmonic functions}. To derive this equality, we first prove that as a consequence of both Cauchy's integral theorem and Cauchy's integral formula, we have that (\ref{Poisson-har}) holds for analytic functions, and from there it will be straightforward that \ref{Poisson-har} holds for harmonic functions. 

\begin{lema}
    If $f$ is an analytic function on $D(0,1+\varepsilon)$ for some $\varepsilon>0$, then \begin{equation}
        f(re^{i\sigma})=\frac{1}{2\pi}\int_{-\pi}^\pi f(e^{it})P_{r}(\sigma-t)\ dt \quad \mbox{ for every } \ re^{i\sigma}\in \D
    \end{equation}
\end{lema}
\begin{proof}
    %Let $\gamma(t):=e^{it}$ for $t\in[-\pi,\pi]$. 
    
    By Cacuhy's integral formula we have $$f(z)=\frac{1}{2\pi i}\int_{\partial\D} \frac{f(w)}{w-z}\ dw\quad \mbox{ for every } \ z\in\D,$$ and since for any $z\in\D$ the function $w\mapsto \frac{f(w)}{w-1/\overline{z}}$ is holomorphic on an open set containing $\overline{\D}$, by Cauchy's integral theorem we have $\int_{\partial\D} \frac{f(w)}{w-1/\overline{z}}\ dw=0$, and thus $$f(z)= \frac{1}{2\pi i}\int_{\partial\D} f(w)\left(\frac{1}{w-z}-\frac{1}{w-1/\overline{z}}\right)\ dw.$$

    Now, if $w\in\partial\D$ then $\frac{1}{w}=\overline{w}$, and computing $$\frac{1}{w-z}-\frac{1}{w-1/\overline{z}}=\frac{(w-1/\overline{z})-(w-z)}{w^2-w(1/\overline{z}+z)+z/\overline{z}}=\frac{z-1/\overline{z}}{w^2-w(1/\overline{z}+z)+z/\overline{z}}=$$ $$=\frac{1}{w}\frac{|z|^2-1}{w\overline{z}-(1+|z|^2)+z/w}=\frac{1}{w}\frac{1-|z|^2}{1-2\Re(z\overline{w})+|z|^2},$$ and we conclude $$f(re^{i\sigma})= \frac{1}{2\pi i}\int_{\partial\D} \frac{f(w)}{w}\frac{1-|z|^2}{1-2\Re(z\overline{w})+|z|^2} \ dw=$$ $$= \frac{1}{2\pi i}\int_{-\pi}^\pi \frac{f(e^{it})}{e^{it}}\frac{1-r^2}{1-2r\cos(\sigma-t)+r^2}ie^{it}\ dw=\frac{1}{2\pi}\int_{-\pi}^\pi f(e^{it})P_{r}(\sigma-t)\ dt.$$
\end{proof}

\begin{remark}
    From this we see that if $f$ is holomorphic on $\D$ and $re^{i\sigma}\in\D$, then for $r<\rho<1$, the function $f_\rho(z):=f(\rho z)$ is holomorphic on $D(0,1/\rho)$ and therefore $$f(re^{i\sigma})=f_\rho((r/\rho)e^{i\sigma})=\frac{1}{2\pi}\int_{-\pi}^\pi f_\rho(e^{it})P_{r/\rho}(\sigma-t)\ dt.$$
\end{remark}
To simplify the notation, for a point $z=re^{i\sigma}\in\D$ we will write $P_z(t):=P_r(\sigma -t)$, and for a function $f$ defined on $\T$, we may consider the function defined on $[-\pi,\pi]$ and write $f(t)$ instead of $f(e^{it})$.

\section{Existence of isometric embedding}

We now proceed to show how the existence of the isometric embedding $H^\infty(\D)\hookrightarrow L^\infty(\T)$ can be deduced from Banach-Alaoglu's theorem, which assures that  any bounded sequence in the dual of a separable normed space has a $w^*$-convergent subsequence.

\begin{theorem}\label{boundary-behaviour}
    Given $f\in H^\infty(\D)$, there is a function $f^*\in L^\infty(\T)$ such that for every $re^{i\sigma}\in\D$ it holds $$f(re^{i\sigma})=\frac{1}{2\pi}\int_{-\pi}^\pi f^*(e^{it})P_r(\sigma-t)\ dt.$$ 
    The mapping $f\mapsto f^*$ defines an isometric algebra homomorphism from $H^\infty(\D)$ into $L^\infty(\T)$.
\end{theorem}
For the proof of the theorem we will need the following lemma:
\begin{lema}
    The linear space generated by $\{P_z\}_{z\in\D}$ is dense in $L^1(\T)$.
\end{lema}
\begin{proof}
    If $\overline{\langle\{P_z\}_{z\in\D}\rangle }\neq L^1(\T)$, by the Hahn-Banach theorem we would find a non-zero function $f\in L^\infty(\T)$ such that $$(f\ast P_r)(\sigma)=\frac{1}{2\pi}\int_{-\pi}^\pi f(e^{it})P_z(t)\ dt=0\quad \mbox{ for every } z=re^{i\sigma}\in \D,$$ and therefore for a fixed $0<r<1$ it is $f\ast P_r\equiv0$. Now, since $L^\infty(\T)\subset L^1(\T)$ and $(P_r)_r$ is an approximate identity for $L^1(\T)$, we have  $||P_r\ast f-f||_1\rightarrow0$ as $r\rightarrow1$, and thus that $f=0$, which is a contradiction. 
\end{proof}
\begin{proof}[Proof of Theorem \ref{boundary-behaviour}]
    Let $f_r(t):=f(rt)\in L^\infty(\T)$ for $0<r<1$. 
    
    The Maximal Modulus Principle implies that $||f_r||\leq||f_s||$ if $0<r\leq s<1$, and thus $||f_r||\leq \lim_{s\rightarrow1}||f_s||=||f||$. Since $L^\infty(\T)\cong(L^1(\T))^*$ and $L^1(\T)$ is separable, we have that $\overline{B_{L^\infty(\T)}(0,||f||)}$ is $w^*$-sequentially compact, and therefore for any sequence $(f_{r_n})_n$ where $(r_n)\subset(0,1)$ with $r_n\rightarrow 1$, there is a subsequence $(f_{r_{n_k}})_k$ and a function $f^*\in \overline{B_{L^\infty(\T)}(0,||f||)}$ such that $f_{r_{n_k}}\xrightarrow{w^*} f^*$.%\frac{r}{r_{n_k}}

    For a fixed $re^{i\sigma}\in\D$, considering any $r<r_{n_k}<1$, we have that $$ f(re^{i\sigma})- \frac{1}{2\pi}\int_{-\pi}^\pi f^*(e^{it})P_r(\sigma-t)\ dt=$$ $$\frac{1}{2\pi}\left(\int_{-\pi}^\pi f_{r_{n_k}}(e^{it})P_{\frac{r}{r_{n_k}}}(\sigma-t)\ dt- \int_{-\pi}^\pi f^*(e^{it})P_r(\sigma-t)\ dt\right)=$$ $$\frac{1}{2\pi}\int_{-\pi}^\pi f_{r_{n_k}}(e^{it})(P_{\frac{r}{r_{n_k}}}-P_r)(\sigma-t)\ dt- \frac{1}{2\pi}\int_{-\pi}^\pi ( f_{r_{n_k}}-f^*)(e^{it})P_r(\sigma-t)\ dt$$ and thus $$\left| f(re^{i\sigma})- \frac{1}{2\pi}\int_{-\pi}^\pi f^*(e^{it})P_r(\sigma-t)\ dt\right|\leq$$ $$\left|\frac{1}{2\pi}\int_{-\pi}^\pi f_{r_{n_k}}(e^{it})(P_{\frac{r}{r_{n_k}}}-P_r)(\sigma-t) dt\right|+ \left|\frac{1}{2\pi}\int_{-\pi}^\pi ( f_{r_{n_k}}-f^*)(e^{it})P_r(\sigma-t) dt\right| $$ $$\leq||f||_\infty ||P_{\frac{r}{r_{n_k}}}-P_r)||_1+ \left|\frac{1}{2\pi}\int_{-\pi}^\pi ( f_{r_{n_k}}-f^*)(e^{it})P_r(\sigma-t)\ dt\right|,  $$ and taking limits as $k\rightarrow \infty$, since $r_{n_k}\rightarrow 1$ and $f_{r_{n_k}}\xrightarrow{w^*} f^*$, we conclude that $$f(re^{i\sigma})= \frac{1}{2\pi}\int_{-\pi}^\pi f^*(e^{it})P_r(\sigma-t)\ dt.$$

    To see that $f\mapsto f^*$ defines an isometric algebra homomorphism from $H^\infty(\D)$ into $L^\infty(\T)$, we first check that $i:H^\infty(\D)\rightarrow L^\infty(\T)$ with $i(f):=f^*$ is well-defined, that is, it does not depend on the sequence $(r_n)_n$. Considering another sequence $(r_n')_n\subset(0,1)$ with $r_n'\rightarrow1$, this would lead to another function $g^*\in L^\infty(\T)$ such that $$f(re^{i\sigma})= \frac{1}{2\pi}\int_{-\pi}^\pi g^*(e^{it})P_r(\sigma-t)\ dt\quad \mbox{ for every } re^{i\sigma}\in\D,$$ but then we would have for every $z=re^{i\sigma}\in\D$ that $$(f^*-g^*)(P_z)=\frac{1}{2\pi}\int_{-\pi}^\pi (f^*-g^*)(e^{it})P_r(\sigma-t)\ dt=f(z)-f(z)=0,$$ and by the previous lemma we would conclude that $f^*=g^*$. That $i$ is a linear homomorphism is straightforward, and since $||f^*|\leq||f||$, it is only left to see that $||f||\leq||f^*||$, but this is clear because $$|f(z)|\leq\frac{1}{2\pi}\int_{-\pi}^\pi |f^*(e^{it})|P_r(\sigma-t)\ dt\leq \frac{||f^*||}{2\pi}\int_{-\pi}^\pi P_r(\sigma-t)\ dt=||f^*||,$$ and thus $||f||=\sup_{z\in\D}|f(z)|\leq||f^*||$
\end{proof}

\section{Extending the embedding to harmonic function}

Since in the proof of Theorem \ref{boundary-behaviour} we only used the fact that $f$ is holomorphic on $\D$ when we claimed that $$f(re^{i\sigma})= \frac{1}{2\pi}\int_{-\pi}^\pi f_{r_{n_k}}(e^{it})P_{r/{r_{n_k}}}(\sigma-t)\ dt,$$ and the same equality holds for complex-valued harmonic function $\D$, it is straightforward that  the isometric algebra embedding can actually be extended to $h^\infty_{\C}(\D)$, the Hardy space of bounded complex-valued harmonic functions on $\D$. 

Moreover, the isometric embedding $h_\C^\infty(\D)\hookrightarrow L^\infty(\T)$ turns out to be surjective: 

\begin{theorem}\label{isometria h(D)}
    If $f$ is a bounded complex-valued harmonic function on $\D$, then there is a function  $f^*\in L^\infty(\T)$ such that for every $re^{i\sigma}\in\D$ it holds $$f(re^{i\sigma})=\frac{1}{2\pi}\int_{-\pi}^\pi f^*(e^{it})P_r(\sigma-t)\ dt.$$ 
    The mapping $f\mapsto f^*$ defines an isometric algebra isomorphism from $h^\infty(\D)$ onto $L^\infty(\T)$.
\end{theorem}

To prove this, we first recall that for an integrable function $f\in L^1(\T)$ its \emph{Poisson integral} is defined by $$F(z):=\frac{1}{2\pi}\int_{-\pi}^\pi f(e^{it})P_z(t)\ dt,$$ and that it satisfies the following: 

%\begin{definition}
%    We define for 
%end{definition}

\begin{proposition}\label{Poisson-int}
    For any real-valued function $f\in L^1(\T)$, its Poisson integral $F(z)$ is an harmonic function on $\D$.
\end{proposition}

\begin{proof}
    Consider the function $g:[-\pi,\pi]\times\D\rightarrow\C$ defined by $g(t,z):=\frac{1+e^{-it}z}{1-e^{-it}z}$, which for each fixed $t$ is holomorphic on $\D$, and define $$G(z):=\frac{1}{2\pi}\int_{-\pi}^\pi g(t,z)f(t)\ dt.$$

    Since $f$ is real-valued, we have that $\Re(G)(z)=\frac{1}{2\pi}\int_{-\pi}^\pi \Re(g(t,z))f(t)\ dt$, and because $$g(t,re^{i\sigma})=\frac{1+re^{-i(\sigma-t)}}{1-re^{-i(\sigma-t)}}\frac{1-re^{-i(t-\sigma)}}{1-re^{-i(t-\sigma)}}=$$ $$=\frac{1-re^{-i(t-\sigma)}+re^{-i(\sigma-t)}-r^2}{1-re^{-i(t-\sigma)}-re^{-i(\sigma-t)}-r^2}=\frac{(1-r^2)+i2\sin(\sigma-t)}{1-2r\cos(\sigma-t)+r^2},$$ we have that $\Re(G)(z)=\frac{1}{2\pi}\int_{-\pi}^\pi P_z(t)f(t)\ dt$, so it is enough to prove that $G(z)$ is holomorphic on $\D$ to show that the Poisson integral of $f$ is harmonic on $\D$.

    Now, considering a closed triangle $\Delta\subset\D$, the function $g(t,z)$ is bounded on $\partial\D\times\partial\Delta$ and $f(t)$ is integrable in $\partial\D$, so the function $g(t,z)f(t)\in L^1(\partial\D\times\partial\Delta)$ and thus by Fubini's theorem $$\int_{\partial\Delta}G(z)\ dz=\int_{\partial\Delta}\left( \frac{1}{2\pi}\int_{-\pi}^\pi g(t,z)f(t)\ dt\right)\ dz=$$$$=\frac{1}{2\pi}\int_{-\pi}^\pi f(t)\left(  \int_{\partial\Delta}g(t,z)\ dz\right)\ dt,$$ but since for a fixed $t $ $g(t,z):\D\rightarrow\C$ is holomorphic, we have that $\int_{\partial\Delta}g(t,z)\ dz=0$, and therefore $\int_{\partial\Delta}G(z)\ dz=0$ for every triangle $\Delta\subset\D$, and an application of Morera's theorem is enough to claim that $G(z)$ is holomorphic on $\D$.
\end{proof}

\begin{proof}[Proof of Theorem \ref{isometria h(D)}]
    From the comments at the beginning of this section it is clear that $f\mapsto f^*$ is an isometric algebra homomorphism from $h^\infty(\D)$ into $L^\infty(\T)$, and by the previous proposition we also have that it is onto. 
\end{proof}

\begin{remark}
    Observe that all this results could be adapted to functions on $\D^n$ with the corresponding Poisson kernels on the polydisks. 
\end{remark}

\end{document}